\documentclass{article}
\usepackage[utf8]{inputenc}
\usepackage{amsmath,amsthm,amssymb}
\usepackage{graphicx}
\usepackage{subfig}
\usepackage{tikz}
\usepackage{fullpage}
\usepackage[style=alphabetic, giveninits=true,maxnames=5]{biblatex}
\usepackage{calrsfs}
\usepackage{epsfig}
\usepackage{dsfont}
\usepackage{mathrsfs}
\usepackage[inline]{enumitem}
\usepackage{todonotes}
\usepackage{mathtools}
\usepackage[absolute,overlay]{textpos}
\usepackage{graphicx}
\usepackage{subfig}
\usepackage{tikz}
\usetikzlibrary{matrix}
\usetikzlibrary{arrows,positioning}
\usepackage{color}
\usetikzlibrary{3d}
\usepackage{xifthen}
\usetikzlibrary{decorations.pathmorphing}
\usepackage{mymacros}
\usepackage[affil-it]{authblk}
\usepackage[style=alphabetic, giveninits=true]{biblatex}
\title{Reversibility,  covariance and coarse-graining for Langevin dynamics: \\ On the choice of multiplicative noise}
\author[1]{Mario Ayala}
\author[2]{Nicolas Dirr}
\author[3]{Grigorios A. Pavliotis}
\author[1]{Johannes Zimmer}

\affil[1]{School of Computation, Information and Technology, Chair for Analysis and Modelling, Technische Universität München, Boltzmannstraße 3, 85748 Garching, Germany}
\affil[2]{School of Mathematics, Cardiff University, Senghennydd Road, Cardiff CF24 4AG, United Kingdom}
\affil[3]{Department of Mathematics, Imperial College London, South Kensington Campus, London SW7 2AZ, United Kingdom}

\date{\today}

\begin{document}
\maketitle

\begin{abstract}
We study the interplay between reversibility, geometry, and the choice of multiplicative noise (in particular It\^{o}, Stratonovich, Klimontovich) in stochastic differential equations (SDEs). Building on a unified geometric framework, we derive algebraic conditions under which a diffusion process is reversible with respect to a Gibbs measure  on a Riemannian manifold. The condition depends continuously on a parameter $\lambda \in [0,1]$ which interpolates between the conventions of It\^o ($\lambda = 0$), Stratonovich ($\lambda = \frac 1 2$)  and Klimontovich ($\lambda = 1$). For reversible slow-fast systems of SDEs with a block-diagonal diffusion structure, we show, using the theory of Dirichlet forms, that both reversibility and the Klimontovich noise interpretation are preserved under coarse-graining.

In particular, we prove that the effective dynamics for the slow variables, 
obtained via projection onto a lower-dimensional manifold, 
retain the Klimontovich interpretation and remain reversible with respect to the marginal Gibbs measure/free energy. 

Our results provide a flexible variational framework for modeling coarse-grained reversible dynamics with nontrivial geometric and noise structures.
\end{abstract}

\section{Introduction}
Diffusion processes, described either  with stochastic differential equations or the corresponding Fokker-Planck equation, are among the fundamental tools of statistical mechanics~\cite{Ris84, Soize1994, Zwan01}. They are used extensively as models in physics, chemistry, biology and many other applications~\cite{coffey04, vanKamp81}. They are also a central tool in the study of algorithms for sampling and optimization, e.g., for Markov Chain Monte Carlo (MCMC) methods, or in the study of stochastic gradient descent~\cite{Li_2019, Dai_Zhu_2020}.  

A natural starting point, both from the perspective of statistical mechanics and from sampling, is to consider the formal Gibbs measure given by
\[
G(dx) = \frac{1}{Z} e^{-V(x)}\, dx,
\]
where $V \colon \mathbb{R}^d \to \mathbb{R}$ is a smooth confining potential and $Z < \infty$ is a normalization constant. We say that the measure is \emph{stationary} or \emph{invariant} for a stochastic process if it is constant under evolution (see Definition~\ref{def.stat-rev-SDE}). A key problem in nonequilibrium statistical mechanics is to characterize the class of diffusion processes for which this measure is not only invariant but also \emph{reversible}~\cite{bakry2013analysis}, i.e.,  the stochastic dynamics satisfies the stronger condition of detailed balance with respect to $G$~\cite[Ch. 4]{pavliotis2014stochastic} (see Section~\ref{sec:rev-stat} for the precise definition).  Then, the thermodynamic equilibrium is encoded not only at the level of invariant measures but also at the level of the pathwise construction of the process~\cite{CoPa_2023}.
  
It is worth highlighting that stationarity and reversibility are distinct notions. While every reversible process admits an invariant measure, the converse is not true: there exist stationary but non-reversible processes for which the dynamics breaks time-reversal symmetry and exhibits strictly positive entropy production. In this case, the distribution of states remains unchanged in time, but typical trajectories and their time-reversals carry different probabilities. As shown in~\cite{maes2003time}, the entropy production can be identified precisely with this asymmetry, and its positivity is a hallmark of nonequilibrium steady states, consistent with fluctuation relations of Gallavotti--Cohen type. Furthermore, recent work~\cite{CoPa_2023} has clarified that for diffusion processes, reversibility depends on a decomposition of the drift into reversible and irreversible parts. In particular, entropy production is determined by how the irreversible component projects onto the range of the diffusion matrix.  

It is well known that the overdamped Langevin stochastic differential equation (SDE)
\[
dX_t = -\nabla V(X_t)\, dt + \sqrt{2}\, dW_t
\]
generates a diffusion process that is reversible with respect to $G$. In this classical setting, the noise is additive, which implies that the volatility matrix $\sigma: \Rd \to \R^{d \times d}$ is constant, here $\sqrt{2}$ times the identity matrix in $\Rd$. The structure of such processes is now well understood, both from the viewpoint of equilibrium statistical mechanics and from applications in sampling and optimization.  

A natural extension is to consider multiplicative noise,
\begin{equation}
\label{eq.lang-intro}
dX_t = -\sigma(X_t)\sigma(X_t)^T \nabla V(X_t)\, dt + \sqrt{2}\, \sigma(X_t)\, dW_t,
\end{equation}
where $\sigma \colon \mathbb{R}^d \to \mathrm{GL}(d) \subset \mathbb{R}^{d \times d}$ is a smooth invertible matrix-valued function and to ask: under what conditions on $\sigma$ is the Gibbs measure $G$ reversible for the dynamics? (The scaling by $\sigma(X_t)\sigma(X_t)^T$ in the drift and $\sqrt{2}\, \sigma(X_t)$ in the noise is natural in view of the fluctuation-dissipation theorems; in particular, for the generator associated with this stochastic process, the drift and diffusion both have a multiplicative prefactor $\sigma^T\sigma$, which can be interpreted as underlying geometry.) This generalization from additive to multiplicative noise is far from a mere technicality. On the one hand, it naturally arises in a variety of modeling situations (see the examples discussed below in this section). On the other hand, it reveals a truly new phenomenon: the presence of state-dependent noise can introduce irreversible components in the dynamics, yet in some cases these can be exactly compensated for by suitable modifications of the drift so that reversibility with respect to $G$ is preserved. Understanding the precise balance between multiplicative noise and drift corrections is therefore central to extending the Langevin paradigm beyond the additive noise setting.  

One motivation for considering the multiplicative case is to enhance the rate of convergence to equilibrium when the process starts from an arbitrary initial distribution. In this context, the speed of convergence is governed by the spectral gap
\[
\lambda_\sigma := \inf_{ \substack{ \phi \\ \int \phi\, dG = 0}} \frac{\mathcal{E}_\sigma(\phi)}{\int \phi^2\, dG},
\]
where $\mathcal{E}_\sigma$ denotes the Dirichlet form associated with the generator of the process. As shown in \cite{lelièvre2024optimizing}, the matrix-valued function $\sigma$ can, in principle, be optimized--subject to appropriate constraints on the diffusion matrix $\sigma \sigma^T$--to maximize the spectral gap and thereby accelerate convergence to equilibrium.  The use of {\emph preconditioned Langevin dynamics} for sampling  is studied in detail in~\cite{Chen_al_2024}; in particular, the connection with Fisher-Rao gradient flow is established and the importance of the affine invariance property is elucidated.

In addition, multiplicative noise appears naturally in the mathematical formulation of fluctuating hydrodynamics, including in models of active matter (see~\cite{Muller2025a} for a derivation). The multiplicative noise structure can also be used to learn macroscopic evolution operators for gradient-flow evolution from particle data~\cite{Li2019a, Huang2025a}. 

The mobility (diffusion) tensor $\sigma(X_t)\sigma(X_t)^T$ defines a differential-geometric structure: first, we notice that the matrix $\sigma$ induces a \emph{heterogeneous noise landscape}, where regions of small $\sigma(x)$ (i.e. small eigenvalues of $ \sigma(x)\sigma(x)^T$) experience suppressed fluctuations, while regions of large $\sigma(x)$ are dominated by noise. This spatial inhomogeneity naturally suggests a geometric perspective, in which $\sigma$ encodes a position-dependent metric structure on $\mathbb{R}^d$. To formalize this intuition, we assume that $\sigma$ is smooth and that the diffusion matrix $\sigma(x)\sigma(x)^T$ is uniformly positive definite. Under these assumptions, the diffusion matrix defines a Riemannian metric on $\mathbb{R}^d$. More precisely, we consider the Riemannian manifold $(\mathbb{R}^d, g)$ with the metric tensor given by $g(x) = \left(\sigma(x) \sigma(x)^T\right)^{-1}$. We endow this manifold with the Levi-Civita connection associated with $g$~\cite{hsu2002stochastic,weinberg1972gravitation}. We summarize the relevant background from differential geometry in Section~\ref{sec:CovariantDiv}. Within this geometric framework, the stochastic process $X_t$ can be interpreted as a diffusion on the Riemannian manifold $(\mathbb{R}^d, g)$.

It is important to emphasize that when the SDE~\eqref{eq.lang-intro} is interpreted in the Itô sense, reversibility with respect to $G$ holds only under strong conditions, see Theorem~\ref{MainThmKlimGeo} and the discussion in Subsection~\ref{sec:main-contrib}. We show that this obstacle can often be overcome by adopting alternative interpretations of the stochastic integral, such as the Stratonovich or Klimontovich conventions, which restore reversibility under suitable structural conditions on $\sigma$. 

Due to the high irregularity of Brownian motion, the meaning of an expression such as~\eqref{eq.lang-intro} depends on how one defines the stochastic integral. This choice, often referred to as the stochastic calculus convention, specifies at which point the integrand is evaluated within each infinitesimal increment of the Riemann sum approximating the integral; see, for example,~\cite[Chap.~3, (3.15)]{pavliotis2014stochastic}. Each convention has its own advantages: It\^o's formulation~\cite{ito1944109} ensures that the integral is a martingale, enabling powerful analytical tools such as martingale problems, which are central in defining weak solutions and proving convergence results; the Stratonovich formulation~\cite{stratonovich1966new} preserves the classical chain rule of differentiation, making it convenient in physical modeling; and the Hänggi--Klimontovich formulation~\cite{klimontovich1990ito,hanggi1978stochastic} (also called anti--It\^o) is often regarded as the most natural from the standpoint of non-equilibrium statistical theory, since as it was recently shown, it links stochastic differential equations with Fokker--Planck equations consistent with Fick's law of diffusion~\cite{escudero2025beneath}. 
Unlike the It\^o integral, these latter formulations involve evaluating the integrand at future (right-endpoint) values of the noise and are therefore anticipating, which introduces additional analytical subtleties in their rigorous construction. Moreover, both the Stratonovich and in particular Hänggi--Klimontovich formulations often provide a stronger mathematical existence theory for multiplicative noise~\cite{grun2025finite}. 
In this paper, we study the problem of identifying reversible measures for the Langevin evolution~\eqref{eq.lang-intro} within this framework of stochastic calculus conventions.

\subsection{Motivating example: the one-dimensional case}
The interplay between the choice of the stochastic integral (noise) and reversibility is already evident in the one-dimensional setting. Consider the measure 
\begin{equation}
    \label{eq.Gibbs}
G(dx) = \frac{1}{Z_V} e^{-V(x)}\, dx,
\end{equation}
of ``Gibbs structure''; here $V \colon \mathbb{R} \to \mathbb{R}$ is a potential (with precise conditions specified later). We ask if $G$ is stationary for a suitable process; a sufficient condition for this is that the process is reversible with respect to this measure. A stationary diffusion process $\{X_t\}_{t \ge 0}$ with invariant law $G$ is called reversible if the time-reversed process $\{X_{T-t}\}_{t\in[0,T]}$ has the same distribution as the forward process $\{X_t\}_{t\in[0,T]}$, for every $T>0$. Equivalently, all finite-dimensional distributions of the reversed process coincide with those of the original one. In this sense, the law of the process is invariant under time reversal, so that trajectories are statistically indistinguishable when run backward in time.
This property implies detailed balance, i.e.\ vanishing stationary probability flux, and zero entropy production rate~\cite[Sec.~4.6]{pavliotis2014stochastic}, \cite{JiangQianQian2004}.

It is well known~\cite[Chap.~I, pag.~46]{bakry2013analysis} that the most general second-order differential operator which is self-adjoint in $L^2(\mathbb{R}, G(dx))$ is given by
\begin{equation}\label{eq:1d_rev_gen_intro}
    \loc f(x) = \left[-\sigma(x)^2 V'(x) + 2 \sigma(x)\sigma'(x)\right] f'(x) + \sigma(x)^2 f''(x),
\end{equation}
if $\sigma(x)$ is a smooth positive function.

We now consider the SDE whose generator is given by~\eqref{eq:1d_rev_gen_intro}. In particular, we consider the generalized SDE
\begin{equation}
\label{eq.SDE-intro-lambda}
dX_t = -\sigma(X_t)^2 V'(X_t)\, dt + \sqrt{2}\, \sigma(X_t)\, \circ_\lambda dW_t,
\end{equation}
where $\circ_\lambda$ denotes a stochastic integral with a parameter $\lambda \in [0,1]$ that controls the evaluation point of the integrand in a discrete Riemann-Stieltjes approximation. In particular, $\lambda = 0$ leads to the It\^o integral, $\lambda = \frac 1 2$ to the Stratonovich integral, and $\lambda=1$ corresponds to the Klimontovich integral; see Appendix~\ref{sec:Conversion}. In this article, de demonstrate that the Klimontovich integral is a more natural choice, when thinking of the thermodynamics and coarse-graining of diffusion processes.

It is well known that the generator for~\eqref{eq.SDE-intro-lambda} is
\begin{equation}
\label{eq.generator-SDE}
\loc_\lambda f(x) = \left[-\sigma(x)^2 V'(x) + 2\lambda\, \sigma(x)\sigma'(x)\right] f'(x) + \sigma(x)^2 f''(x).
\end{equation}
A comparison of~\eqref{eq:1d_rev_gen_intro} and~\eqref{eq.generator-SDE} shows that in the It\^o formulation, $\lambda=0$, there is no stochastic process of the form~\eqref{eq.SDE-intro-lambda} which renders the process as reversible (and hence stationary) under the measure~\eqref{eq.Gibbs}, unless $\sigma$ is constant. In contrast, in the Klimontovich interpretation, $\lambda = 1$, the process is always reversible with respect to~\eqref{eq.Gibbs}. In fact, the
 Klimontovich noise is the only interpretation for which the canonical Gibbs measure is in one space dimension reversible under multiplicative noise. 

It is tempting to extrapolate from this example and conclude that the Klimontovich interpretation automatically yields reversibility for arbitrary $\sigma$ in any finite dimension. However, this is not the case. In higher dimensions, additional algebraic conditions on the matrix field $\sigma$ must be imposed in order to ensure reversibility. This stems from the fact that the Euclidean divergence operator $\nabla \cdot$ does not satisfy a chain rule, as would be required for the cancellation of noise-induced drifts. As a result, reversibility under Klimontovich noise (and for other choices of noise) in higher dimensions depends sensitively on the structure of $\sigma$, and must be characterized by explicit geometric conditions. We give such a characterisation in Theorem~\ref{MainThmKlimGeo}.

\subsection{Main task}\label{sec.main_task}
Let $V\colon \mathbb{R}^d \to \mathbb{R}$ be a smooth confining potential such that
\begin{equation}\label{Vgrowsfast}
  \lim_{|x| \to \infty} V(x) = +\infty,
\end{equation}
and $e^{-\beta V(x)} \in L^1(\mathbb{R}^d)$ for $\beta >0$.

Let $\sigma \colon \mathbb{R}^d \to \mathrm{GL}(d) \subset \mathbb{R}^{d \times d}$ be a smooth invertible matrix-valued function and define the diffusion tensor by
\[
M(x) := \sigma(x)\sigma(x)^T.
\]
We assume that $\sigma$ is \emph{uniformly elliptic}; that is, there exists a constant $C > 0$ such that
\[
\xi^T M(x)\, \xi \geq C \|\xi\|^2 \quad \text{for all } x \in \mathbb{R}^d \text{ and } \xi \in \mathbb{R}^d.
\]
The matrix field $M(x)$ induces a Riemannian metric $g$ with $g(x) = M^{-1}(x) = \left(\sigma(x)\sigma^T(x)\right)^{-1}$, turning the ambient space $ \mathbb{R}^d $ into a Riemannian manifold $ (\mathbb{R}^d, g) $. In this geometric setting, the Gibbs measure is defined as
\begin{equation}
\label{eq.gibbs-g}
G(dx) = \frac{1}{Z_V} e^{-V(x)}\, \mathrm{vol}_M(dx),
\end{equation}
where $ \mathrm{vol}_M(dx) := \sqrt{\omega_M(x)} dx$ is the Riemannian volume element corresponding to $g$, where $\omega_M(x) := \det(M^{-1}(x))$ is the Riemannian volume density.

We consider in this manifold a class of stochastic differential equations of the form
\begin{equation}\label{eq: noise-Lange}
dX_t = -\sigma(X_t)\sigma(X_t)^T \, \nabla V(X_t)\, dt + \sqrt{2}\, \sigma(X_t)\, \circ_\lambda dW_t,
\end{equation}
where $ \nabla V $ denotes the Euclidean gradient, and $ \circ_\lambda $ denotes a stochastic integral interpreted with parameter $ \lambda \in [0,1] $, interpolating between the conventions of Itô ($ \lambda = 0 $), Stratonovich ($ \lambda = \tfrac{1}{2} $), and Klimontovich ($ \lambda = 1 $).

We can also consider the inverse problem~\cite{lelièvre2024optimizing}: Given a smooth confining potential $V$, determine all smooth, uniformly elliptic diffusion matrices $M \colon \mathbb{R}^d \to \mathbb{R}^{d \times d}$ and noise characterized by $\lambda \in [0,1]$ such that the dynamics \eqref{eq: noise-Lange} is reversible with respect to the Riemannian Gibbs measure~\eqref{eq.gibbs-g}. That is, the inverse problem is to identify the pairs $(M, \lambda)$ for which the generator associated with \eqref{eq: noise-Lange} is (essentially) self-adjoint in $L^2(G)$, or equivalently, for which the dynamics satisfies the detailed balance condition with respect to $G$ on the manifold $(\mathbb{R}^d, g)$. This inverse problem can also be studied in a data-driven framework: we can learn the diffusion matrix--as well as the nonreversible perturbation and the entropy production rate--from particle trajectories~\cite{zhu2025identifiablelearningdissipativedynamics}.

\subsection{Main contributions}
\label{sec:main-contrib}
The main contributions of this paper can be summarized as follows.

\begin{itemize}
    \item We derive explicit algebraic conditions, in Euclidean local coordinates, on the matrix $\sigma$ under which the SDE \eqref{eq: noise-Lange} is reversible, depending on the parameter $\lambda$. In particular:
    \begin{itemize}
        \item \emph{Itô noise:} Reversibility holds if and only if 
        \begin{equation}\label{eq: geomcond ito}
            \nablac \cdot (\sigma\sigma^T) = 0,
        \end{equation}
        where $\nabla^c$ denotes the covariant divergence operator on matrix fields with respect to the Levi-Civita connection induced by $g$, see Section~\ref{sec:CovariantDiv}. By the metric compatibility this condition is equivalent to the canonical Euclidean coordinates being harmonic (see~\eqref{eq.harmonic}) for the metric $g = (\sigma\sigma^T)^{-1}$. This condition is rather restrictive: While it holds trivially for constant $\sigma$, i.e., additive noise, multiplicative It\^o noise results in a reversible diffusion process (with respect to the given Gibbs measure) only if the covariant divergence vanishes. In one space dimension, this conditions can never be satisfied unless $\sigma$ is constant;  for general state-dependent $\sigma$ it typically fails. Thus It\^o diffusions with multiplicative noise generically break reversibility and produce entropy, even if the stationary distribution is (formally) Gibbsian.
        \item \emph{Stratonovich noise:} Reversibility holds if and only if
        \begin{equation}\label{eq: geomcond strat}
            \sigma \nablac \cdot \sigma^T = 0.  
        \end{equation}
        In one space dimension, this condition cannot be satisfied, unless $\sigma$ is constant. However, in general, this is a weaker condition for reversibility than the one for It\^o noise, since to satisfy condition~\eqref{eq: geomcond strat} it is sufficient for the combination of $\sigma$ and its covariant derivative to vanish. 
        \item \emph{Klimontovich noise:} Reversibility holds if and only if
        \begin{equation}\label{eq: geomcond klim}
            \nablac \cdot(\sigma \sigma^T) = 2 \sigma \nablac \cdot \sigma^T.
        \end{equation}
        In one space dimension, this condition is always satisfied.
    \end{itemize}
    For other values $\lambda \in(0,1]$, conditions similar to~can be derived for $\lambda = \frac 1 2 $ in~\eqref{eq: geomcond strat} and $\lambda =1$ in~\eqref{eq: geomcond klim} ensuring reversibility.
    
    \item We demonstrate that, in the context of averaging principles for certain reversible slow-fast systems with multiplicative noise, both the reversibility and the Klimontovich interpretation of noise are preserved under averaging/coarse-graining. 
    
    \item We establish that averaging principles for fast-slow systems of the structure~\eqref{eq: noise-Lange} can be rigorously formulated via Mosco convergence of Dirichlet forms in the sense of Kuwae and Shioya~\cite{kuwae2003convergence}. This provides a rigorous approach to coarse-graining in this context, providing an alternative, elegant formulation to multiscale analysis and two--sclae convergence~\cite{DDP_2023}. In our setting, the averaging corresponds to a projection onto an effective submanifold, preserving the variational structure of the limit dynamics.
\end{itemize}

\paragraph{Organization of the paper.}
Section~\ref{sec:Preliminaries} reviews reversibility and stationarity for SDEs and introduces the geometric framework. In particular, Section~\ref{sec:CovariantDiv} recalls the covariant divergence and compares it with its Euclidean counterpart. Our main result, the derivation of algebraic conditions for reversibility depending on the noise interpretation, is stated and proved in Section~\ref{sec:MainResult}. There, we also interpret these conditions for the noises in the sense of Itô, Stratonovich, and Klimontovich. In Section~\ref{sec:Averaging}, we show that both reversibility and the Klimontovich interpretation of noise are preserved under averaging, using Mosco convergence of Dirichlet forms. In Appendix~\ref{sec:Conversion}  we recall the conversion formula for different noises and in Appendix~\ref{appendix.mosco} we collect background material on Mosco convergence.

\section{Preliminaries}\label{sec:Preliminaries}

\subsection{Reversibility and stationarity}
\label{sec:rev-stat}

In the analysis of stochastic processes, two fundamental notions characterize equilibrium behavior: \emph{stationarity} and \emph{reversibility}. Both can be understood either at the level of trajectories, via stochastic differential equations, or at the level of probability distributions, via the Fokker--Planck equation--at least for Markov processes.   Though closely related, they express different aspects of equilibrium: the former concerns the invariance of distributions under time evolution, while the latter reflects the symmetry of the dynamics under time reversal.

\subsubsection{Trajectory (SDE) perspective}

Let $\{X_t\}_{t\ge0}$ be a diffusion process on a smooth Riemannian manifold $(\mathcal{M},g)$, whose infinitesimal generator
\[
    \loc : \mathcal{D}(\loc) \subset C(\mathcal{M}) \to C(\mathcal{M})
\]
acts on sufficiently smooth test functions. From
 this point of view, stationarity and reversibility can be formulated directly in terms of~$\loc$.

\begin{definition}
\label{def.stat-rev-SDE}
The process admits a \emph{stationary measure}~$\mu$ if
\[
    \int_{\mathcal{M}} (\loc f)(x)\,\mu(dx) = 0
    \qquad\text{for all } f\in\mathcal{D}(\loc).
\]
This ensures that statistical observables do not evolve in time when the initial distribution is~$\mu$.

The process is \emph{reversible} with respect to~$\mu$ if the generator~$\loc$ is symmetric in~$L^2(\mu)$:
\[
    \int_{\mathcal{M}} f(x)\,(\loc g)(x)\,\mu(dx)
    = \int_{\mathcal{M}} g(x)\,(\loc f)(x)\,\mu(dx),
    \qquad \forall f,g\in\mathcal{D}(\loc).
\]
This condition encodes the invariance of path probabilities under time reversal.
\end{definition}

\subsubsection{Fokker--Planck (distributional) perspective}

The same notions can be expressed at the level of evolving probability measures.  
Let $(\mathcal{M}, g)$ be as above, with Riemannian volume form~$\mathrm{vol}$.

\medskip
Denote by $\mathbb{M}(\mathcal{M})$ the space of finite signed (Radon) measures on~$\mathcal{M}$,
and by $\mathscr{D}'(\mathcal{M})$ the space of distributions, i.e., the continuous dual of
$C_c^\infty(\mathcal{M})$.

For the diffusion process with generator~$\loc$, we define the adjoint operator
\[
  \loc^*:\mathbb{M}(\mathcal{M}) \longrightarrow \mathscr{D}'(\mathcal{M}),
  \qquad
  \int_{\mathcal{M}} (\loc f)\,d\mu
  = \int_{\mathcal{M}} f\,d(\loc^*\mu),
  \quad f\in\mathcal{D}(\loc),
\]
which acts on finite signed measures~$\mu$.

If $\mu_t$ denotes the law of $X_t$, its evolution is governed by the
\emph{forward Kolmogorov equation} (interpreted in some appropriate weak sense, see~\cite{BKRS2015})
\[
  \partial_t \mu_t = \loc^* \mu_t .
\]
When $\mu_t$ is absolutely continuous with respect to~$\mathrm{vol}$, we denote its density (Radon-Nikodim derivative)
\[
  p_t := \frac{d\mu_t}{d\mathrm{vol}}.
\]
There exists a (density–form) operator $\FP_{df}$ on functions such that
\[
  \loc^*(p\,\mathrm{vol}) = (\FP_{df}p)\,\mathrm{vol},
\]
and the evolution above becomes the \emph{Fokker--Planck equation}
\[
  \partial_t p_t = \FP_{df} p_t .
\]
Equivalently, in weak form,
\[
  \frac{d}{dt}\int_{\mathcal{M}} f\,p_t\,d\mathrm{vol}
  = \int_{\mathcal{M}} (\loc f)\,p_t\,d\mathrm{vol}
  = \int_{\mathcal{M}} f\,(\FP_{df} p_t)\,d\mathrm{vol},
  \qquad f\in\mathcal{D}(\loc).
\]

\smallskip
For simplicity of notation, we use the same symbol~$\loc^*$ to denote both the adjoint acting on measures and its corresponding action on densities, whenever no confusion arises.

If, in addition, $\loc^*$ admits a divergence representation on measures,
\[
  \loc^*\mu = -\,\nabla\!\cdot J(\mu),
\]
then, for densities~$p$, this takes the continuity form
\[
  \partial_t p + \nabla\!\cdot J(p) = 0,
\]
where $J(p)$ is the (probability) current associated with~$\FP_{df}$ through $\FP_{df}p = -\nabla\!\cdot J(p)$.

\begin{definition}
\label{def.stat-rev-FP}
A probability measure $\mu$ on~$\mathcal{M}$ is \emph{stationary} for~$\loc$ if
\[
    \loc^*\mu = 0,
\]
that is, $\int \loc f\,d\mu = 0$ for all $f\in\mathcal{D}(\loc)$.  
Equivalently, if $p = \frac{d\mu}{d\mathrm{vol}}$, then $\partial_t p = 0$ in the Fokker--Planck equation.

The process is \emph{reversible} with respect to~$\mu$ if the stationary current vanishes,
\[
    J(\mu) = 0,
\]
which implies the detailed balance condition
\[
    \int f\,(\loc g)\,d\mu = \int g\,(\loc f)\,d\mu,
    \qquad \forall f,g\in \mathcal{D}(\loc).
\]
\end{definition}

\medskip
Every reversible process is stationary, but the converse need not hold.  
These two viewpoints: trajectory-based and distribution-based, are complementary.  
In geometric settings, one must ensure that divergence and adjoint operations are taken with respect to the correct volume form, but the conceptual distinction remains the same.

\subsection{The generator of the Langevin SDE}
Let $X_t$ be given by the solution of the Langevin SDE \eqref{eq: noise-Lange} (we assume that $\sigma$ and $V$ satisfy the assumptions of Section~\ref{sec.main_task} ensuring the existence of a solution and the well-definedness of the Gibbs measure $G$). We recall that this structure induces a metric $g$, via  $M(x) = \left(\sigma(x)\sigma^T(x)\right)$ and $g = M^{-1}$. In this setting, see, for example,~\cite[Pag.~314]{capitaine2007onsager}, the infinitesimal generator of this diffusion process takes the form
\begin{equation}\label{eq: inf_gen_geo}
\loc_M f = \Delta_M f + B \cdot \nabla f,
\end{equation}
where $\Delta_M$ is the Laplace–Beltrami operator associated with the metric $g$ given in~\eqref{eq: def DeltaM} below, and $B$ denotes a corrected drift vector field that accounts for both the original drift $-M \nabla V$, the correction due to the noise (such as the Stratonovich-It\^o correction in case of Stratonovich noise, $\lambda = \frac 1 2$, see Appendix~\ref{sec:Conversion}), and geometric contributions arising from the stochastic calculus on the manifold, see~\eqref{eq.drift}.

In local coordinates, the Laplace–Beltrami operator acting on a smooth function $f \colon \R^d \to \R$ reads
\begin{equation}\label{eq: def DeltaM}
\Delta_M f = \frac{1}{\sqrt{\omega_M}}\partial_i \left(  \sqrt{\omega_M} \, M^{ij}\, \partial_j f \right),
\end{equation}
where $M^{ij}$ denotes the $(i,j)$-entry of the diffusion metric (or the inverse metric) $M(x) = \sigma(x) \sigma^T(x)$, and $\omega_M(x) := \det(M(x)^{-1})$ is the Riemannian volume density, see~\cite[Appendix C]{bakry2013analysis}. Notice that we have used the Einstein convention for summation of repeated indices. In local coordinates, using the noise correction given by~\eqref{DriftDictionary}, the corrected drift vector takes the form
\begin{equation}
\label{eq.drift}
    B^i =\underbrace{ - M^{ij} (\nabla V)_j}_{\text{drift}} + \underbrace{2 \lambda \left( \partial_j M^{ji} - \sigma^{i\ell} \partial_k \sigma^{k\ell}  \right)}_{\text{noise term}} + \underbrace{M^{kj} \Gamma_{kj}^i,}_{\text{geometric correction}}
\end{equation}
where $\Gamma_{ij}^k$ denote the Christoffel symbols, see Subsection~\ref{sec:CovariantDiv}, whose term in~\eqref{eq.drift} comes from the relation
\begin{equation*}
    \Delta_M f =  M^{ij} \partial_i \partial_j f - M^{kj} \Gamma_{kj}^i \partial_i f.
\end{equation*}
This relation can also be found  in~\cite[Pag.~314]{capitaine2007onsager}.

\subsection{The Fokker--Planck equation}

When considering $X_t$ as taking values in the Riemannian manifold $\left(\R^d, g\right)$, it is natural to describe the evolution of the law of $X_t$ in terms of a probability density with respect to the Riemannian volume measure
\[
d\mu_M(x) := \sqrt{\omega_M(x)}\, dx.
\]
 Let us denote this density by $K(t,x)$, so that the law of $X_t$ satisfies
\[
\mathbb{P}(X_t \in A) = \int_A K(t,x) \, d\mu_M(x).
\]
Equivalently, we may relate $K(t,x)$ to the \emph{flat} probability density $p(t,x)$ through
\[
K(t,x) = p(t,x) \, \frac{1}{\sqrt{\omega_M(x)}}.
\]

Since the Laplace--Beltrami operator $\Delta_M$ is self-adjoint in $L^2(d\mu_M)$, the Fokker--Planck equation for $K(t,x)$ takes the covariant form
\begin{equation}\label{eq:fpe_covariant}
\partial_t K = -\nablac_i \left( B^i K \right) + M^{ij} \nablac_i \nablac_j K,
\end{equation}
where the Einstein summation convention is used and $\nablac_i$ denotes the covariant derivative with respect to the metric $g$. We refer to~\cite{graham1977covariant,jauslin1985classification} for a derivation. 

This formulation ensures full compatibility with the Riemannian structure induced by the noise, and coincides with the covariant Fokker--Planck equation derived by Graham in the context of nonequilibrium statistical mechanics and stochastic thermodynamics~\cite{graham1977covariant}.

\subsection{Covariant derivative and covariant divergence}\label{sec:CovariantDiv}

Let $\left(\Rd, g\right)$ be the Riemannian manifold where the metric $g$ is defined as the inverse of the diffusion matrix $M(x) = \sigma(x) \sigma^T(x)$. In local coordinates, we denote by $g_{ij}(x)$ the components of the metric and by $M^{ij}(x)$ the components of its inverse.

The Levi--Civita connection $\nablac$, often called \emph{covariant derivative}, is the unique torsion-free connection compatible with the metric. Based on Theorem 4.31 and Definition 4.31 in~\cite{jostRiemannianGeometryGeometric2011} we adopt the following definition (formulae for the Levi--Civitta connection acting on tensors are given below in Proposition~\ref{prop.cov-deriv}).

\begin{definition}[Levi--Civita connection]
The \emph{Levi--Civita connection} $\nablac$ on $\left(\Rd,g\right)$ is the unique affine connection, that is:
\begin{enumerate}
    \item It is \emph{torsion-free}, 
    \[
    \nablac_X Y - \nablac_Y X = [X,Y]
    \]
    for all vector fields $X,Y$, where $[X,Y]:=XY-YX$ denotes the Lie bracket.
    \item It is \emph{compatible with the metric}, in the sense that $\nablac g = 0$, i.e.,
    \begin{equation}\label{metcompact}
        \nablac_\ell g_{ij} = 0 \quad \text{for all} \quad i,j,\ell \in \{1,\dots,d\}.
    \end{equation}
\end{enumerate}
These conditions imply that the parallel transport preserves the Riemannian inner product, and that raising and lowering indices commutes with covariant differentiation. In particular, we have
\begin{equation}\label{eq: inv metric comp}
 \nablac_\ell M^{ij} = 0   
\end{equation}
for all $i,j,\ell \in \{1,\dots,d\}$.
\end{definition}

From the definition of the Levi--Civita connection we can derive, see Corollary 4.3.1 in \cite{jostRiemannianGeometryGeometric2011}, an expression for the corresponding \emph{Christoffel symbols}, namely
\begin{equation*}
    \Gamma_{ij}^k = \frac12\, g^{k\ell} \left( \partial_i g_{j\ell} + \partial_j g_{i\ell} - \partial_\ell g_{ij} \right).
\end{equation*}Note that these symbols encode how coordinate basis vectors change from point to point. Moreover we have the following.

\begin{proposition}[Trace of the Christoffel symbols]\label{prop:gamma-trace}
We recall that $\omega_M= \det(M^{-1})$ is the Riemannian volume density.  
The contracted Christoffel symbols satisfy
\begin{equation}\label{eq:gamma-trace}
  \Gamma^{k}_{\,k j}
  \;=\; \tfrac12\,\partial_j \log\omega_M
  \;=\;  \frac{1}{\sqrt{\omega_M}}\, \partial_j \big(\sqrt{\omega_M}\big).
\end{equation}
\end{proposition}

\begin{proof}
This identity is a standard property of the Levi--Civita connection; see display~(4.7.6) in~\cite[Sec.~4.7]{weinberg1972gravitation}.
\end{proof}

From the Levi--Civita definition and the expression for the Christoffel symbols, one obtains the following formulas in coordinates for the covariant derivative acting on tensor fields.

\begin{proposition}[Coordinate expression of the covariant derivative]
\label{prop.cov-deriv}
For the Levi--Civita connection $\nablac$ associated with $(\mathbb{R}^d,g)$, the covariant derivative acts on tensor fields as follows:
\begin{align}
    \nablac_k \, f &= \partial_k f, & &\text{for a scalar function $f$}, \label{eq.cov-sca} \\
    \nablac_k \, X^j &= \partial_k X^j + \Gamma_{kl}^j X^l, & &\text{for a vector field $X = X^j \partial_j$}, \label{eq.cov-tens10} \\
    \nablac_k \, A^{ij} &= \partial_k A^{ij} + \Gamma_{k l}^i A^{l j} + \Gamma_{k l}^j A^{il}, & &\text{for a $(2,0)$-tensor $A$},
    \label{eq.cov-tens20} 
\end{align}
\end{proposition}

\begin{proof}
These formulas are standard coordinate expressions for the Levi--Civita connection; see, for example,~\cite[Section~4.7]{weinberg1972gravitation}. 
In particular, equation~\eqref{eq.cov-sca} corresponds to display~(4.7.1) in~\cite{weinberg1972gravitation}, equation~\eqref{eq.cov-tens10} corresponds to display~(4.7.3), and equation~\eqref{eq.cov-tens20} appears in the unnumbered display between~(4.7.8) and~(4.7.9).
\end{proof}

\begin{remark}
For a general $(p,q)$-tensor $T$ we can also obtain a coordinate expression for its covariant derivative. Namely
\begin{equation}\label{eq.cov-tenspq}
 (\nablac_k T)^{i_1 \dots i_p}{}_{j_1 \dots j_q}
    = \partial_k T^{i_1 \dots i_p}{}_{j_1 \dots j_q}
    + \sum_{\alpha=1}^{p} \Gamma_{k l}^{i_\alpha} T^{i_1 \dots l \dots i_p}{}_{j_1 \dots j_q}
    - \sum_{\beta=1}^{q} \Gamma_{k j_\beta}^{l} T^{i_1 \dots i_p}{}_{j_1 \dots l \dots j_q}.   
\end{equation}
We refer to~\cite[Chap.~4,Prop.~4.18]{lee2018introduction} for more details.
\end{remark}

\begin{definition}[Row covariant divergence of a matrix field]\label{def: cov div}
Let $A = \left(A^{j\ell}\right)_{1 \le j,\ell \le d}$ be a smooth matrix field on $(\mathbb{R}^d,g)$, i.e., a smooth map $x \mapsto A(x)$ whose entries $A^{j\ell}(x)$ are real-valued functions. 
The \emph{covariant divergence} of $A$ is defined as the natural generalization of the column-wise divergence of a matrix in Euclidean space, obtained by contracting the covariant derivative over the matrix second (column) index:
\begin{equation}\label{eq: def cov div}
    (\nablac \cdot A)_j := \partial_i A^{ji} + \Gamma_{ik}^i A^{j k} = (\nabla \cdot A)_j + \Gamma_{ik}^i A^{jk},
\end{equation}
where $\nabla \cdot A$ is the Euclidean divergence of $A$ as defined in \cite[Appendix~A,pag.~300]{pavliotis2014stochastic}. 
In Euclidean coordinates with the standard (Cartesian) metric, all \emph{Christoffel symbols} vanish, and this reduces to the divergence taken along the rows of $A$.
\end{definition}

\begin{remark}\label{RemarkCovDivVsDeriv}
In the Euclidean setting, that is, when the Cartesian metric is used, the divergence of a smooth matrix field $A$ can be defined either \emph{column-wise} or \emph{row-wise}. 
For instance, in~\cite[Sec.~1.2]{bronasco2025efficient} the divergence of $A$ is defined as the vector whose $j$th component is the divergence of the $j$th \emph{column} of $A$,
\begin{equation*}
    (\nabla \cdot A)_j = \partial_i A^{ij},
\end{equation*}
where the Einstein summation convention is used.
Alternatively, in~\cite[Appendix~A, p.~300]{pavliotis2014stochastic} the divergence of $A$ is defined \emph{row-wise} as
\begin{equation}\label{eq.rowwisecovdiv}
    (\nabla \cdot A)_j = \partial_i A^{ji}.
\end{equation}
This latter convention is particularly convenient, since it allows one to express the drift corrections associated with different noise interpretations in a compact matrix form (see~\eqref{DriftDictionary} in Appendix~\ref{sec:Conversion}).

In the geometric setting, i.e., for an arbitrary metric, the notion of divergence depends on the tensor type of $A$. 
As discussed in~\cite[Chap.~1, Def.~4.27]{marsden1994mathematical} or more abstractly in \cite[Chap.~5,Prob.~5-16]{lee2018introduction}, if $A$ is a $(p,q)$--tensor, its \emph{covariant divergence} is the $(p-1,q)$--tensor obtained by contracting the last contravariant index of its covariant derivative,
\begin{equation}\label{eq:def.covdiv.marsden}
    (\nablac \cdot A)^{i_1 \dots i_{p-1}}{}_{j_1 \dots j_q}
    := (\nablac_{i_p} A)^{i_1 \dots i_p}{}_{j_1 \dots j_q},
\end{equation}
where $(\nablac_{i_p} A)^{i_1 \dots i_p}{}_{j_1 \dots j_q}$ is given by~\eqref{eq.cov-tenspq}, again following the Einstein summation convention. 

Our motivation for introducing a distinct \emph{row covariant divergence} stems from the fact that the standard covariant divergence~\eqref{eq:def.covdiv.marsden} vanishes identically for the co-metric tensor $M = \sigma \sigma^\top$ (viewed as a $(2,0)$--tensor) due to the metric-compatibility condition~\eqref{eq: inv metric comp}. 
To capture the correct geometric analogue of the Euclidean (row-wise) divergence used in~\cite{pavliotis2014stochastic}, we therefore define the \emph{row covariant divergence} of a $(2,0)$--tensor $A$ by viewing each of its rows $A^{j(\cdot)}$ as a vector field and contracting its covariant derivative accordingly, namely
\begin{equation*}
    (\nablac \cdot A)_j := \nablac_i (A^{j(\cdot)})^i.
\end{equation*}
\end{remark}

The next result shows that our \emph{row covariant divergence} reproduces exactly the compact form of the noise correction as used in~\cite[Chap.~3,pag.~63]{pavliotis2014stochastic}: in the key combination below, all Christoffel contributions cancel.

\begin{proposition}\label{prop:row-div-compat}
Let $\sigma=(\sigma^{i\ell})$ be a smooth field with one contravariant manifold index $i\in\{1,\dots,d\}$ and one auxiliary Euclidean index $\ell\in\{1,\dots,m\}$, and let $M=\sigma\sigma^\top$ with components $M^{ij}=\sum_{\ell}\sigma^{i\ell}\sigma^{j\ell}$. 
Then, componentwise for $j=1,\dots,d$,
\begin{equation}\label{eq:correction}
\bigl[\nablac\!\cdot(\sigma\sigma^\top)\bigr]_j \;-\; \bigl[\sigma\,(\nablac\!\cdot\sigma^\top)\bigr]_j
\;=\;
\bigl[\nabla\!\cdot(\sigma\sigma^\top)\bigr]_j \;-\; \bigl[\sigma\,(\nabla\!\cdot\sigma^\top)\bigr]_j,
\end{equation}
where $\nablac \cdot$ is the row covariant divergence from Definition~\ref{def: cov div}, and $\nabla \cdot$ is the \emph{Euclidean} row-wise divergence of \eqref{eq.rowwisecovdiv}. 
\end{proposition}

\begin{proof}
Using Einstein summation, expand
\[
\bigl[\nablac\!\cdot(\sigma\sigma^\top)\bigr]_j
= \nablac_i(\sigma^{j\ell}\sigma^{i\ell})
= \partial_i(\sigma^{j\ell}\sigma^{i\ell}) + \Gamma^i_{\,ik}\sigma^{j\ell}\sigma^{k\ell}
= (\partial_i\sigma^{j\ell})\sigma^{i\ell} + \sigma^{j\ell}\partial_i\sigma^{i\ell} + \Gamma^i_{\,ik}\sigma^{j\ell}\sigma^{k\ell}.
\]
Also,
\[
\bigl[\sigma\,(\nablac\!\cdot\sigma^\top)\bigr]_j
= \sigma^{j\ell}\bigl(\partial_i\sigma^{i\ell}+\Gamma^i_{\,ik}\sigma^{k\ell}\bigr)
= \sigma^{j\ell}\partial_i\sigma^{i\ell} + \Gamma^i_{\,ik}\sigma^{j\ell}\sigma^{k\ell}.
\]
Subtracting, both $\sigma^{j\ell}\partial_i\sigma^{i\ell}$ and the Christoffel trace terms $\Gamma^i_{\,ik}\sigma^{j\ell}\sigma^{k\ell}$ cancel, leaving
\[
\bigl[\nablac\!\cdot(\sigma\sigma^\top)\bigr]_j - \bigl[\sigma\,(\nablac\!\cdot\sigma^\top)\bigr]_j
= (\partial_i\sigma^{j\ell})\sigma^{i\ell}.
\]
In Euclidean (Cartesian) coordinates, $\nabla\!\cdot$ is just $\partial_i$ acting row-wise, hence
\[
\bigl[\nabla\!\cdot(\sigma\sigma^\top)\bigr]_j
= \partial_i(\sigma^{j\ell}\sigma^{i\ell})
= (\partial_i\sigma^{j\ell})\sigma^{i\ell} + \sigma^{j\ell}\partial_i\sigma^{i\ell},
\quad
\bigl[\sigma\,(\nabla\!\cdot\sigma^\top)\bigr]_j
= \sigma^{j\ell}\partial_i\sigma^{i\ell}.
\]
Subtracting yields the same remainder $(\partial_i\sigma^{j\ell})\sigma^{i\ell}$, which proves \eqref{eq:correction}.
\end{proof}

\section{Algebraic conditions for reversibility} 
\label{sec:MainResult}
The reversibility of the stochastic process obeying~\eqref{eq: noise-Lange} with respect to a given invariant measure imposes strong structural constraints on the interplay between the drift $-\sigma(x)\sigma^T(x) \nabla V$, the volatility $\sigma(x)$  and the choice $\lambda$ of the stochastic integral/noise. The following theorem provides an explicit algebraic criterion under which reversibility holds. 

\begin{theorem}\label{MainThmKlimGeo}
Let $M=\sigma\sigma^T$ and $g=M^{-1}$. Consider the Gibbs measure
\[
G(dx)=\frac{1}{Z_V}\,e^{-V(x)}\,d\mu_M(x),
\]
where 
$d\mu_M(x):=\sqrt{\omega_M(x)}\,dx$ and $\omega_M(x):=\det(M(x)^{-1})$.
Let $X_t$ solve
\begin{equation}\label{KlimSDEGeo}
dX_t=-\,M(X_t)\,\nabla V(X_t)\,dt+\sqrt{2}\,\sigma(X_t)\circ_{\lambda} dW_t.
\end{equation}
Then $G$ is reversible for $X_t$ (i.e., the generator is self-adjoint in $L^2(G)$) if and only if
\begin{equation}\label{eq: geom cons}
(2\lambda-1)\,\nabla^{c}\!\cdot (\sigma \sigma^T) \;=\; 2\lambda\,\sigma\,(\nabla^{c}\!\cdot \sigma^{T}),
\end{equation}
where the covariant divergences are those of Definition~\eqref{def: cov div}.
\end{theorem}

\begin{proof}
The proof proceeds by comparing the generator $\loc$ of the geometric SDE~\eqref{KlimSDEGeo} with the reversible generator $\loc_G$ associated with the Gibbs measure~$G(dx)$.  
For a given volatility matrix~$\sigma$, all diffusion generators $\loc$ are completely determined (see~\cite[Chap.~1]{bakry2013analysis}).  
Using Proposition~\ref{prop:row-div-compat}, we express $\loc$ in covariant form—including the noise–interpretation correction—and verify~\eqref{eq: geom cons} by direct term-by-term comparison with~$\loc_G$.

\smallskip
\textbf{Step 1. Reversible generator.}
Using ~\cite[Chap.~1, pag.~47]{bakry2013analysis}, we know that the Gibbs measure $G(dx) \propto e^{-V(x)}\sqrt{\omega_M(x)}\,dx$ is reversible for the diffusion process with generator
\begin{equation*}
\loc_G f
=\Delta_M f + \Gamma(-V,f),
\end{equation*}
where $\Delta_M$ is the Laplace--Beltrami operator associated with the metric $M^{-1}$, and $\Gamma$ is the carré–du–champ operator,
\[
\Gamma(f,g) := \frac12\!\left[\loc(fg) - f\,\loc g - g\,\loc f\right],
\qquad
\Gamma(-V,f) =- M^{ij}\,(\partial_j V)\,\partial_i f.
\]
Using~\eqref{eq:gamma-trace}, the generator $\loc_G$  admits the coordinate representation
\begin{equation}\label{eq:L_G_coord}
\loc_G f
= M^{ij}\,\partial_{ij}^2 f
 -\,\Gamma_{ik}^{\,j}M^{ik} \,\partial_j f
- M^{ij}\,(\partial_i V)\,\partial_j f.
\end{equation}

\smallskip
\textbf{Step 2. Generator of the SDE.}
Let $\loc$ denote the infinitesimal generator of the geometric SDE~\eqref{KlimSDEGeo} including the $\lambda$--Itô correction. By the noise correction formula~\eqref{DriftDictionary} we have that $\loc$ can be written as  
\begin{equation*}
\loc_G f
= M^{ij}\,\partial_{ij}^2 f
+ \Bigl[ 2\lambda\,\bigl(\nabla \cdot(\sigma\sigma^T) - \sigma\,\nabla \cdot\sigma^T\bigr)_j\Bigr]\partial_j f
- M^{ij}\,(\partial_i V)\,\partial_j f.
\end{equation*}
Moreover, by Proposition~\ref{prop:row-div-compat} and the notation in Definition~\ref{def: cov div}, this generator can be written in covariant form as
\begin{equation}\label{eq:L_cov}
\loc f
= M^{ij}\,\partial_{ij}^2 f
+ \Bigl[ 2\lambda\,\bigl(\nablac\!\cdot(\sigma\sigma^T) - \sigma\,\nablac\!\cdot\sigma^T\bigr)_j\Bigr]\partial_j f - M^{ij}\,(\partial_i V)\,\partial_j f.
\end{equation}

\smallskip
\textbf{Step 3. Identification of the geometric correction.}
By the inverse metric compatibility~\eqref{eq: inv metric comp} and~\eqref{eq.cov-tens20},
\[
0=\nablac_\ell M^{ij}
=\partial_\ell M^{ij}+\Gamma_{\ell k}^{\,i}M^{kj}+\Gamma_{\ell k}^{\,j}M^{ik},
\]
which, after contracting $\ell=i$, gives
\begin{equation}\label{eq.cov.div.contract}
-\,\Gamma_{ik}^{\,j}M^{ik}
= \partial_i M^{ij}+\Gamma^i_{\,ik}M^{kj}
= (\nablac\!\cdot(\sigma\sigma^T))_j.
\end{equation}
Hence, the covariant correction term in~\eqref{eq:L_G_coord} coincides with the geometric divergence of $M$. That is, we can write 
\begin{equation}\label{eq:L_G_coord_cov}
\loc_G f
= M^{ij}\,\partial_{ij}^2 f
 + (\nablac\!\cdot(\sigma\sigma^T))_j \,\partial_j f
- M^{ij}\,(\partial_i V)\,\partial_j f.
\end{equation}
\smallskip
\textbf{Step 4.}
Comparing~\eqref{eq:L_G_coord_cov} and~\eqref{eq:L_cov} term by term, and using the fact that reversibility of~$G(dx)$ requires $\loc=\loc_G$, which holds if and only if
\[
2\lambda\!\left[\nablac\!\cdot(\sigma\sigma^T) - \sigma\,\nablac\!\cdot\sigma^T\right]
=\nablac\!\cdot(\sigma\sigma^T),
\]
that is, condition~\eqref{eq: geom cons}.
\end{proof}

\subsection{Interpretation of the algebraic conditions}
In the It\^o interpretation ($\lambda=0$), the reversibility condition~\eqref{eq: geom cons} reduces to
\begin{equation}
    \nablac \cdot M = 0,
\end{equation}
with $M = \sigma \sigma^T$. By relation~\eqref{eq.cov.div.contract}, this is equivalent to
\begin{equation}
\label{eq.harmonic}
 \Gamma^j_{ik}M^{ik} = 0,   
\end{equation}
for all $j \in \{1,\dots,d\}$. In the physics literature this is known as the \emph{de Donder gauge condition}, while in differential geometry it is the requirement that the coordinates are \emph{harmonic}. In other words, reversibility with respect to the Gibbs measure requires that the canonical Euclidean coordinates are harmonic for the metric $g=M^{-1}$ induced by $\sigma$. This equivalence between It\^o reversibility and harmonic coordinates was already emphasized by Graham in his covariant analysis of nonequilibrium diffusions~\cite{graham1977covariant}. The conclusion is that the It\^o convention is overly restrictive for modeling reversible systems with multiplicative noise: unless the Euclidean chart happens to be harmonic for $g$ (a very strong condition, satisfied in particular when $\sigma$ is constant), It\^o diffusions with state-dependent noise generically break reversibility and produce entropy even when the stationary distribution is Gibbs.

In the Klimontovich interpretation, $\lambda = 1$, the reversibility condition~\eqref{eq: inv metric comp} becomes
\begin{equation}\label{eq: geom Klim cond}
  \nablac \cdot(\sigma \sigma^T) = 2  \sigma \nablac \cdot \sigma^T.
\end{equation}
This identity always holds in dimension one and, of course, whenever $\sigma$ is constant. In higher dimensions ($d > 1$), it defines a nontrivial algebraic constraint on $\sigma$. A broad class of reversible systems arises when $\sigma(x)$ is diagonal, i.e.,
\[
\sigma(x) = \operatorname{diag}(\sigma_1(x), \ldots, \sigma_d(x)),
\]
in which case~\eqref{eq: geom Klim cond} is automatically satisfied. More generally, the condition holds for block-diagonal matrices where each block independently satisfies~\eqref{eq: geom Klim cond}, and remains valid under constant orthogonal changes of basis: if
\[
\sigma(x) = U D(x) U^T,
\]
with $U$ orthogonal and constant and $D(x)$ block-diagonal as above, then condition~\eqref{eq: inv metric comp} still holds. Physically, this corresponds to systems where the directions of independent noise are globally defined and do not vary with position. In contrast, if $U$ depends on $x$, the condition generally fails, as the rotation of the noise basis introduces geometric distortions which cannot be compensated. Thus, reversibility under Klimontovich noise is closely linked to the existence of a global orthonormal frame in which the noise acts independently.

In the Stratonovich interpretation, $\lambda = \frac{1}{2}$, the reversibility condition reduces to
\begin{equation}\label{eq: geom Strato cond}
 \sigma \nablac \cdot \sigma^T = 0.
\end{equation}
This condition is again trivially satisfied when $\sigma$ is constant. More generally,~\eqref{eq: geom Strato cond} imposes a geometric constraint on the divergence of the noise field: it requires that each column of $\sigma$ (viewed as a vector field) be divergence-free when projected back through $\sigma$. This can be interpreted as a form of \emph{incompressibility} of the noise, ensuring that the noise-induced drift does not generate any net expansion or contraction of volume in the dynamics. Similar conditions have appeared in the theory of covariant stochastic differential equations in manifolds, particularly in the work of Graham and collaborators \cite{graham1977covariant, graham1985covariant, risken1972solutions}, and more recently in \cite{diosi2024covariant}. In these contexts, condition~\eqref{eq: geom Strato cond} ensures the compatibility between stochastic dynamics and underlying geometric structure.

\subsection{Covariant formulation of It\^o SDEs}
In~\cite{graham1977covariant}, Graham considered an SDE with noise interpreted in the Stratonovich sense and derived a covariant formulation of the associated Fokker-Planck equation. We recall that the Stratonovich differential satisfies the chain rule of differentiation and hence can be seen as a covariant object.

In~\cite{graham1985covariant}, Graham pointed out that the It\^o differential is not covariant (due to the second order correction in It\^o's lemma), and introduced a new re-interpretation of the It\^o SDE that made it ``manifestly covariant''. More precisely, for an It\^o SDE of the form  
\begin{equation}\label{eq: Ito SDE B}
    d X_t = B(X_t)\, dt + \sqrt{2}\, \sigma(X_t) \circ_I d W_t,
\end{equation}
 Graham considered the SDE (written component-wise form for convenience) given by
\begin{equation}\label{eq: SDE Graham in Ito01}
    d X_t^\mu = \left( B^\mu(X_t)\, - \sqrt{\omega_M} \partial_\alpha \left(\frac{M^{\alpha \mu}}{\sqrt{\omega_M}} \right) \right)dt + \sqrt{2}\, \sigma^{\mu\nu} (X_t) \circ_I d W_t^\nu,
\end{equation}
and showed that this new SDE is covariant.

We write
\begin{equation}\label{eq: SDE Graham noise}
    d X_t = B(X_t)\, dt + \sqrt{2}\, \sigma(X_t) \circ_G d W_t,
\end{equation}
to refer to the SDE~\eqref{eq: SDE Graham in Ito01} and use the symbol $\circ_G$ to indicate that the noise is interpreted in the sense of Graham. By this, we mean that the SDE~\eqref{eq: SDE Graham noise} can be seen as an alternative interpretation of noise that corresponds to the It\^o SDE 
\begin{equation}\label{eq: SDE Graham in Ito}
    d X_t = \tilde{B}(X_t)\, dt + \sqrt{2}\, \sigma(X_t) \circ_I d W_t,
\end{equation}
where the Graham-It\^o correction results in the drift $\tilde{B}^\mu=B^\mu - \sqrt{\omega_M} \partial_\alpha \left(\frac{M^{\alpha \mu}}{\sqrt{\omega_M}} \right)$. 

\begin{remark}
We do not claim here that there exists a direct definition (e.g., using stochastic integrals involving possibly spatially inhomogeneous~$\lambda$) of the noise~$\circ_G$. 
Nevertheless, the Graham noise is well defined through its Itô representation.
\end{remark}

In a recent paper~\cite{diosi2024covariant}, under the condition that the parameter $\sigma$ is composed of divergence-free smooth vector fields, Di\'osi found that the Stratonovich SDE 
\begin{equation}\label{eq: SDE diosi with Graham}
    d X_t = \tilde{B}(X_t)\, dt + \sqrt{2}\, \sigma(X_t) \circ_S d W_t
\end{equation}
agrees with the It\^o SDE \eqref{eq: Ito SDE B}. 
More precisely, the covariant condition assumed in~\cite{diosi2024covariant} which in our setting is imposed on the condition 
\begin{equation}\label{diosicond}
    \nablac \cdot \sigma^T = 0,
\end{equation}
implies that the SDE~\eqref{eq: SDE diosi with Graham} and the SDE~\eqref{eq: Ito SDE B} represent the same stochastic evolution. 

The basic idea of~\cite{diosi2024covariant} is that, under condition~\eqref{diosicond}, the Stratonovich-It\^o correction matches the correction term in the drift of the SDE originally imposed by \cite{graham1977covariant} to make the dynamics covariant. That is, for all $1 \leq j \leq d$ 
\begin{equation}\label{Diosimatchingterms}
  \left(  \nabla \cdot (\sigma(X_t) \sigma^T(X_t)) - \sigma(X_t) \nabla \cdot \sigma^T(X_t)  \right)_j =  \sqrt{\omega_M}  \partial_j \left( \frac{1}{\sqrt{\omega_M}} M^{ij}  \right). 
\end{equation}
We now show that the weaker condition 
\begin{equation}\label{StratoSDEcondGeo}
  \sigma \nablac \cdot \sigma^T = 0
\end{equation}
is already enough to conclude~\eqref{Diosimatchingterms}, and as a consequence to show that the Stratonovich SDE~\eqref{eq: SDE diosi with Graham}
agrees with the original It\^o SDE~\eqref{eq: Ito SDE B}.
In order to verify that this is indeed the case, it is enough to develop the right hand side of \eqref{Diosimatchingterms} as follows:
\begin{align}\label{RHSofdiosimatching}
\sqrt{\omega_M}  \partial_j \left( \frac{1}{\sqrt{\omega_M}} M^{ij}  \right) &=   \partial_j  M^{ij} +   M^{ij} \sqrt{\omega_M} \partial_j \left( \frac{1}{\sqrt{\omega_M}}\right) \nn \\
  &=   \partial_j  M^{ij} +   M^{ij}  \Gamma_{kj}^k \nn \\
  &= \nablac \cdot (\sigma(X_t) \sigma^T(X_t))_j,
\end{align}
where in the second line we used~\eqref{eq:gamma-trace} and in the last line~\eqref{eq: def cov div} and the symmetry of $M$.

The fact that the last line of~\eqref{RHSofdiosimatching} corresponds to the Stratonovich-It\^o correction under the weaker condition~\eqref{StratoSDEcondGeo} justifies our claim.\\

\begin{remark}
Notice that this is also true for any other interpretation of noise different from It\^o's noise, as long as we impose the algebraic conditions for reversibility analogous to~\eqref{eq: geom cons}. However, as explained in~\cite{diosi2024covariant}, only the Stratonovich interpretation makes the left-hand side of the SDE (i.e., the stochastic differential) satisfy a chain-rule.
\end{remark}

\section{Averaging via Dirichlet forms}\label{sec:Averaging}

We now show that the Klimontovich noise is preserved under coarse-graining of suitable  fast-slow systems with multiplicative (nonconstant) noise; this result ensures that the macroscopic reversible measure retains a Gibbsian structure compatible with the microscopic model. Specifically, we consider slow-fast systems where the noise sources driving the slow and fast variables are independent and the dynamics are initially reversible with Klimontovich noise. Our approach combines ideas from the projection operator method (in the spirit of Mori (see, e.g.,~\cite{Widder2025aTR}) with the framework of Mosco convergence of Dirichlet forms developed by Kuwae and Shioya~\cite{kuwae2003convergence}. Perhaps surprisingly, the abstract setting of Kuwae and Shioya aligns remarkably well with the projection operator approach, allowing us to rigorously justify that both the interpretation of noise and the reversible structure are preserved under averaging. 
The approach developed here uses advanced tools, notably Mosco convergence, but is overall simple and provides an alternative to a derivation by asymptotic expansion or two-scale convergence~\cite{DDP_2023}.  

\subsection{One-dimensional case}

We consider a prototypical slow-fast system which is structurally of the form~\eqref{eq.SDE-intro-lambda}, where now the fast dynamics are accelerated relative to the slow ones. Specifically, we rescale the noise intensity and drift in the fast variable $Y$ by a factor of $\sqrt{n}$ and $n$, respectively, where $n \gg 1$. The resulting system reads
\begin{align}
    d X_t^n &= - \sigma_1^2(X_t^n,Y_t^n) \, \partial_x V(X_t^n,Y_t^n) \, dt + \sqrt{2}\,\sigma_1(X_t^n,Y_t^n) \circ_K d W_t^1, \label{eq.Xn}\\
    d Y_t^n &= - n\, \sigma_2^2(X_t^n,Y_t^n) \, \partial_y V(X_t^n,Y_t^n) \, dt + \sqrt{2 n} \, \sigma_2(X_t^n,Y_t^n) \circ_K d W_t^2. \label{eq.Yn}
\end{align}
Here, both noise terms are interpreted in the Klimontovich sense, and the slow  variable $X$ and the fast variable $Y$ are coupled through the potential $V$ and the coefficients $\sigma_1$ and $\sigma_2$. In this case, the volatility matrix is diagonal,
\begin{equation*}
    \sigma(x,y) = \begin{bmatrix}
        \sigma_1(x,y) & 0 \\
        0 & \sigma_2(x,y)
    \end{bmatrix}.
\end{equation*}
This form ensures that the noise sources acting on $X_t$ and $Y_t$ are independent.  By Theorem~\ref{MainThmKlimGeo}, for any $C^2$-smooth coefficients $\sigma_1,\sigma_2>0$ the reversibility condition \eqref{eq: geom Klim cond} holds component-wise and the Gibbs measure
\begin{equation}
\label{eq:Gibbs-xy}
\mu(dx\,dy)=Z_V^{-1}\,e^{-V(x,y)}\,dx\,dy
\end{equation}
is thus reversible for~\eqref{eq.Xn}--\eqref{eq.Yn}, uniformly in $n$.

Throughout this section, we make the following assumptions.
\begin{assumption}
The potential $V \colon \mathbb{R}^2 \to \mathbb{R}$ and the coefficients $\sigma_1, \sigma_2 \colon \mathbb{R}^2 \to (0,\infty)$ are $C^2$ and $V$ is confining in the sense that $Z_{V,\beta} := \int_{\mathbb{R}^2} \exp(-\beta V(x,y))\,dx\,dy < \infty$ for all $\beta > 0$. 

Furthermore, we assume that for each fixed $x \in \mathbb{R}$, the fast process
\[
dY_t = -\sigma_2^2(x,Y_t)\,\partial_y V(x,Y_t)\,dt + \sqrt{2}\,\sigma_2(x,Y_t)\circ_K dW_t^2
\]
is ergodic with unique invariant measure $\mu^x(dy) \propto \exp(-V(x,y))\,dy$.
\end{assumption}

These conditions ensure that the SDE system is well-posed, the associated weighted Sobolev spaces are Hilbert spaces, which in turn implies that the associated Dirichlet forms are closed~\cite{fukushima2010dirichlet}. In the literature, it is typically required that the volatility coefficients $\sigma_1$ and $\sigma_2$ are locally integrable, to ensure the completeness of the associated weighted Sobolev spaces; see, for example,~\cite{zhikov1998weighted}. In our setting, this condition is automatically satisfied due to the $C^2$-regularity of $\sigma_1$ and $\sigma_2$.

\subsection{Averaging and preservation of Klimontovich structure}
\label{sec:slow-fast-gen}
It is a classical result, see~\cite{khasminskij1968principle, PavlSt08},  and~\cite{duong2025coarsegrainingstochasticdifferential} for an alternative approach, that in the limit $n \to \infty$, the fast variable $Y_t^n$ equilibrates rapidly, and the slow component $X_t^n$ converges to an effective process $\bar{X}_t$ governed by the SDE
\begin{equation}
\label{eq.effSDE}
    d \bar{X}_t = \bar{b}(\bar{X}_t) \, dt + \sqrt{2} \, \bar{\sigma}_1(\bar{X}_t) \, d W_t,
\end{equation}
where the effective drift and volatility coefficients are given by
\begin{align}
    \bar{b}(x) &= \frac{1}{Z_V(x)} \int_\R \left( \partial_x \sigma_1^2(x,y) - \sigma_1^2(x,y) \, \partial_x V(x,y) \right) \exp(-V(x,y)) \, dy \label{eq.eff.b} \intertext{and}
    \bar{\sigma}_1(x) &= \left( \frac{1}{Z_V(x)} \int_\R \sigma_1^2(x,y) \exp(-V(x,y)) \, dy \right)^{1/2},
    \label{eq.eff.sigma}
\end{align}
with the normalization $Z_V(x) = \int_\R \exp(-V(x,y)) \, dy$ arising as the marginal of the original Gibbs measure with respect to $y$. For more details, we refer to~\cite{khasminskij1968principle} (the explicit expressions for $b$ and $\sigma$ rely on the ergodicity in the fast variable).

A key observation is that the limiting dynamics for $\bar{X}_t$ can also be written in Klimontovich form. Indeed, differentiating $\bar{\sigma}_1^2(x)$ given by~\eqref{eq.eff.sigma} and using~\eqref{eq.eff.b} yields
\begin{equation*}
    \partial_x \left( \bar{\sigma}_1^2(x) \right) = \bar{b}(x) + \bar{\sigma}_1^2(x) \frac{\partial_x Z_V(x)}{Z_V(x)}.
\end{equation*}
This allows us to recast the effective SDE~\eqref{eq.effSDE} as
\begin{equation*}
    d \bar{X}_t = -\bar{\sigma}_1^2(\bar{X}_t) \, \partial_x \ln \left( \int_\R \exp(-V(\bar{X}_t,y)) \, dy \right) \, dt + \sqrt{2} \, \bar{\sigma}_1(\bar{X}_t) \circ_K d W_t.
\end{equation*}

This computation shows in a transparent manner that both reversibility and the interpretation of noise are preserved under averaging: the limiting process is again reversible with respect to the marginal Gibbs measure
\begin{equation}
    \label{eq.gibbs-marg}
    \mu_\infty(dx) \propto \int_\R \exp(-V(x,y))\, dy,
\end{equation}
and the effective SDE retains the Klimontovich structure.

\subsection{Averaging via Dirichlet forms}
\label{SecWarmUpDirforms}

Beyond the classical approach using generators as presented in Section~\ref{sec:slow-fast-gen}, the averaging principle can also be rigorously justified within the Dirichlet form framework. This perspective is particularly powerful, as it connects probabilistic dynamics to variational structures and provides a natural setting for taking singular limits such as $n \to \infty$. The theory of Dirichlet forms was applied to the study of averaging--in particular in connection to the Freidlin-Wentzell limit--in~\cite{BarretRenesse2014}.

Due to reversibility, see, for example~\parencite[(1.11.8)]{bakry2013analysis}, the infinitesimal generator of the two-dimensional process~\eqref{eq.Xn}--\eqref{eq.Yn} takes the form
\begin{equation*}
    \loc_n F(x,y) = e^{V(x,y)} \sum_{i,j=1}^2  \partial_i \left( e^{-V(x,y)} M_n^{ij}(x,y) \partial_{j} F(x,y) \right),
\end{equation*}
where the diffusion matrix is
\begin{equation}
\label{eq.Mn}
    M_n(x,y)= \begin{bmatrix}
        \sigma_1^2(x,y) & 0 \\
        0 & \sqrt{n} \,  \sigma_2^2(x,y)
    \end{bmatrix}.
\end{equation}

Instead of working only with the generator, it is often advantageous—especially for rigorous convergence results—to consider the associated Dirichlet forms. For each $n$, the Dirichlet form $(E_n, D(E_n))$ on $L^2(\R^2, \mu)$ is given by
\begin{align*}
    E_n(F,G) &= -\int_{\R^2} G(x,y) \, \loc_n F(x,y) \, \mu(dx\,dy), \intertext{with}
    D(E_n) &= \left\{ F \in L^2(\R^2,\mu) \bigm| E_n(F,F) < \infty \right\},
\end{align*}
where $\mu(dx\,dy)$ is the reversible Gibbs measure.

In the averaging limit, we expect the slow variable $\bar{X}_t$ to evolve according to an effective Dirichlet form $(E, D(E))$ on $L^2(\R, \mu_\infty)$, where
\begin{align*}
    E(f) &= -\int_{\R} f(x) \, \loc f(x) \, \mu_\infty(dx), \intertext{with}
    D(E) &= \left\{ f \in L^2(\R, \mu_\infty) \bigm| E(f) < \infty \right\},
\end{align*}
and, using the results of \cite{khasminskij1968principle}, the limiting generator is
\begin{align*}
\loc f(x) = \left( -\bar{\sigma}_1^2(x) \, \partial_x \ln{\left(\int_\R \exp{(-V(x,y))}\, dy\right)} + \partial_x ( \bar{\sigma}_1^2(x) )\right) f'(x) + \bar{\sigma}_1^2(x) f''(x).
\end{align*}

\subsection{Convergence of Hilbert spaces}

A subtlety arises because the Dirichlet forms $E_n$ and $E$ are defined on different spaces: functions of $(x,y)$ versus functions of $x$ alone. The rigorous connection between them is established using the notion of Mosco convergence, as developed by Kuwae and Shioya~\cite{kuwae2003convergence}. This framework allows us to systematically compare Dirichlet forms on varying Hilbert spaces and to justify the convergence of the slow-fast system to its averaged limit not just at the level of trajectories, but also at the level of variational (energetic) structures.
For the slow-fast system studied here, the sequence of Dirichlet forms $(E_n, D(E_n))$ acts on functions of both $(x, y)$, while the limiting Dirichlet form $(E, D(E))$ acts on functions of $x$ alone. This mismatch of domains requires a careful analysis of the convergence of the underlying Hilbert spaces.

For each $n$, we consider the weighted Sobolev space
\begin{equation*}
H_n^1(\mu):= \left\{ f \in L^2(\R^2,\mu) \bigm| \int_{\R^2} \sum_{ij} M_n^{ij} (\partial_i f) (\partial_j f) \,  \mu(dx,dy) < \infty \right\}, 
\end{equation*}
where $M_n$, given in~\eqref{eq.Mn}, encodes the anisotropic volatility coefficients. The limiting space is the weighted Sobolev space
\begin{equation*}
H^1(\mu_\infty):= \left\{ f \in L^2(\R,\mu_\infty) \bigm| \int_{\R}  \bar{\sigma}_1^2(x) (\partial_x f)^2  \mu_\infty(dx) < \infty \right\}, 
\end{equation*}
where $\bar{\sigma}_1^2(x)$ is the effective (averaged) volatility coefficient and $\mu_\infty$ is the marginal Gibbs measure~\eqref{eq.gibbs-marg} for $x$.

Both $H_n^1(\mu)$ and $H^1(\mu_\infty)$ are Hilbert spaces when equipped with the inner products
\begin{align*}
    \langle f, g \rangle_{ H_n^1(\mu)} &=  \langle f, g \rangle_{L^2(\mu)} + E_n(f,g) \intertext{and}
    \langle f, g \rangle_{ H^1(\mu_\infty)} &=  \langle f, g \rangle_{L^2(\mu_\infty)} + E(f,g),
\end{align*}
respectively. To connect these spaces, we define the linear embedding $\Phi_n\colon H^1(\mathbb{R},\mu_\infty) \to H_n^1(\mathbb{R}^2,\mu)$ by
\begin{equation*}
    (\Phi_n f)(x,y) = f(x),
\end{equation*}
so that $\Phi_n f$ is constant in $y$. This operator allows us to compare functions in the limiting space with their ``lifts'' to the pre-limit space.

\begin{remark}\label{ProjDirformagrees}
A key property is that $\Phi_n$ is an isometry with respect to both the $L^2$ and $H^1$ norms. Furthermore, for all $f \in H^1(\mathbb{R},\mu_\infty)$, by~\eqref{eq.eff.sigma} and~\eqref{eq.gibbs-marg}, 
\begin{align*}
  E_n(\Phi_n f) &= \int_{\mathbb{R}^2} \sigma_1^2(x,y) \, (\partial_x f(x))^2 \, \frac{1}{Z_V} \exp(-V(x,y)) \, dx \, dy \\
  &= \int_{\mathbb{R}} (\partial_x f(x))^2 \, \bar{\sigma}_1^2(x) \, \mu_\infty(dx) \\
  &= E(f),
\end{align*}
where $\bar{\sigma}_1^2(x)$ is the effective volatility coefficient and $\mu_\infty$ is the marginal measure. Thus, the structure of the limiting space is naturally embedded in the pre-limit spaces via $\Phi_n$.
\end{remark}

This observation leads directly to the following proposition.

\begin{proposition}\label{ConvHilSpaceWarmup}
The sequence of Hilbert spaces $\lbrace H_n^1(\mu) \bigm| n \in \N \rbrace$ converges, in the sense of Kuwae and Shioya~\cite{kuwae2003convergence}, to the Hilbert space $H^1(\mu_\infty)$, with the embedding given by $\Phi_n$.
\end{proposition}

\begin{proof}
By construction, $\Phi_n$ is an isometry in $H^1$, so for all $f \in H^1(\mu_\infty)$ and all $n$, i.e.,
\[
    \|\Phi_n f\|_{H_n^1(\mu)} = \|f\|_{H^1(\mu_\infty)}.
\]
\end{proof}

\subsection{Convergence of Dirichlet forms}

Given the convergence of the underlying Hilbert spaces, we can now address the convergence of the Dirichlet forms themselves. 

\begin{theorem}
The sequence of Dirichlet forms $\lbrace (E_n, D(E_n)) \bigm| n \in \N \rbrace$ converges, in the sense of Mosco (see Appendix~\ref{appendix.mosco}), to the Dirichlet form $(E,D(E))$.
\end{theorem}

\begin{proof}
Since the embedding $\Phi_n$ is an isometry and the domains are compatible, the conditions for Mosco convergence are satisfied (see Theorem~\ref{KolesnikovThm}). In particular, strong and weak convergence under $\Phi_n$ correspond to the usual notions in the respective spaces, and the forms agree on embedded functions.
\end{proof}

\subsection{Generalization of the averaging principle to higher dimensions}

The averaging principle and Mosco convergence arguments presented above for the two-dimensional slow-fast system extend naturally to higher-dimensional settings. Here, we briefly indicate how the key ideas and results carry over.

Consider a $(d+m)$-dimensional system of SDEs of the form
\begin{align*}
    d X_t^{(n)} &= - \sigma_1(X_t^{(n)}, Y_t^{(n)}) \sigma_1^T(X_t^{(n)}, Y_t^{(n)}) \nabla_x V(X_t^{(n)}, Y_t^{(n)}) \, dt + \sqrt{2}\, \sigma_1(X_t^{(n)}, Y_t^{(n)}) \circ_K d W_t^1, \\
    d Y_t^{(n)} &= - n\, \sigma_2(X_t^{(n)}, Y_t^{(n)}) \sigma_2^T(X_t^{(n)}, Y_t^{(n)}) \nabla_y V(X_t^{(n)}, Y_t^{(n)}) \, dt + \sqrt{2n}\, \sigma_2(X_t^{(n)}, Y_t^{(n)}) \circ_K d W_t^2,
\end{align*}
where $X_t^{(n)}$ is $d$-dimensional, $Y_t^{(n)}$ is $m$-dimensional, and the noises in the $X$ and $Y$ components remain independent.

\begin{proposition}
Suppose the joint process $(X_t^{(n)}, Y_t^{(n)})$ is reversible with reversible measure
\begin{equation*}
    \mu_n(dx\, dy) = \frac{1}{Z_V} \exp\big(-V(x,y)\big)\, dx\, dy.
\end{equation*}
Then the matrices $\sigma_1$ and $\sigma_2$ satisfy the reversibility condition~\eqref{eq: geom cons}.
\end{proposition}

\begin{proof}
By reversibility, see~\parencite[(1.11.8)]{bakry2013analysis}, the infinitesimal generator of the process is given by
\begin{equation*}
    \loc_n F(\bx,\by) = e^{V(\bx,\by)} \sum_{i,j=1}^{d+m} \partial_i \left( e^{-V(\bx,\by)} M_n^{ij}(\bx,\by) \partial_j F(\bx,\by) \right),
\end{equation*}
where, due to the independence of the noises $W^1$ and $W^2$, the matrix $M_n(\bx,\by)$ is block diagonal,
\begin{align*}
    M_n(\bx,\by) = \begin{bmatrix}
        \sigma_1(\bx,\by) \sigma_1^T(\bx,\by) & 0 \\
        0 & n\, \sigma_2(\bx,\by) \sigma_2^T(\bx,\by)
    \end{bmatrix}.
\end{align*}
The block-diagonal structure ensures that the argument reduces to the one-dimensional case for each component. More precisely,
\begin{align*}
    &\begin{bmatrix}
        \nabla \cdot \left(\sigma_1(\bix,\biy) \sigma_1^T(\bix,\biy)\right) & 0 \\
        0 & n \,  \nabla \cdot \left( \sigma_2(\bix,\biy) \sigma_2^T(\bix,\biy)\right)
    \end{bmatrix}
   \nonumber \\
   &= \begin{bmatrix}
       2 \sigma_1(\bix,\biy) & 0 \\
        0 & 2 \sqrt{n} \,  \sigma_2(\bix,\biy)
    \end{bmatrix}
    \begin{bmatrix}
        \nabla \cdot \left(\sigma_1^T(\bix,\biy) \right) & 0 \\
        0 & \sqrt{n} \,  \nabla \cdot \left(\sigma_2^T(\bix,\biy)\right)
    \end{bmatrix},
\end{align*}
and a block-wise comparison enables us to apply Theorem~\ref{MainThmKlimGeo} with $\lambda=1$ to each block. This shows that both $\sigma_1$ and $\sigma_2$ must satisfy the reversibility condition~\eqref{eq: geom cons} in their respective dimensions. 
\end{proof}

All the main constructions and results from the two-dimensional case (Section~\ref{SecWarmUpDirforms}) extend to this higher-dimensional setting. In particular, the effective diffusion parameter for the averaged dynamics is given by
\begin{equation*}
    \frac{1}{Z_V(\bx)} \int_{\R^m} \sigma_1(\bx,\by) \sigma_1^T(\bx,\by) \exp(-V(\bx,\by))\, d\by,
\end{equation*}
where the integral is taken in the sense of Bochner for matrix-valued functions. The convergence of Dirichlet forms and the preservation of reversibility then follow in a straightforward manner.

\subsection{Interpretation of noise for the averaged dynamics}
To conclude, we illustrate in detail how the Klimontovich noise is preserved under averaging, by considering a concrete family of volatility coefficients. Our aim is to construct an example where the Klimontovich reversibility condition~\eqref{eq: geom cons},
\begin{equation}\label{GoldenProp}
   \nabla \cdot \left(  \sigma_1(\bx) \, \sigma_1^T(\bx) \right) = 2  \sigma_1(\bx) \,  \nabla \cdot   \sigma_1^T(\bx),
\end{equation}
continues to hold after averaging over the fast variables.

In analogy to~\eqref{eq.eff.sigma}, the effective volatility matrix for the slow variables, after averaging, can be written as any square root of the Bochner integral, in the sense that
\begin{equation*}
  \bar{\sigma}_1(\bx) \, \bar{\sigma}_1^T(\bx) =  \frac{1}{Z_V(\bx)}   \int_{\R^m}  \sigma_1(\bx,\by) \, \sigma_1^T(\bx,\by) \, \exp(-V(\bx,\by)) \, d\by.
\end{equation*}
We want to ensure that the property \eqref{GoldenProp} is preserved for $\bar{\sigma}_1(\bx)$. To this end, let us consider the class of matrices
\begin{equation*}
    \sigma_1(\bx,\by) = U \Lambda(\bx,\by) U^T,
\end{equation*}
where $U = (u_{ij})$ is a constant orthogonal matrix and $\Lambda(\bx,\by)$ is diagonal. This structure, as discussed earlier, guarantees that $\sigma_1$ satisfies \eqref{GoldenProp} before averaging.

For this form, a natural choice for the square root of the averaged matrix, motivated by~\eqref{eq.eff.sigma}, is
\begin{equation*}
  \bar{\sigma}_1(\bx) :=  U \left(  \int_{\R^m}  \Lambda^2(\bx,\by) \, \frac{1}{Z_V(\bx)} \exp(-V(\bx,\by)) \, d\by \right)^{1/2} U^T.
\end{equation*}
Our goal is to verify that 
\begin{equation}\label{VeriKlimCond}
   \nabla \cdot \left(  \bar{\sigma}_1(\bx) \, \bar{\sigma}_1^T(\bx) \right) = 2  \bar{\sigma}_1(\bx) \,  \nabla \cdot   \bar{\sigma}_1^T(\bx),
\end{equation}
which implies by Theorem~\ref{MainThmKlimGeo} that the averaged dynamics is reversible as well. 

Let us now go through the explicit computation. We begin by computing the divergence of the transpose of $\bar{\sigma}_1(\bx)$:
\begin{align*}
    \left(\nabla \cdot \bar{\sigma}_1(\bx)^T \right)_i 
    &= \sum_{j} \partial_{x_j} \sum_{\ell} u_{j\ell} \left(  \int_{\R^m}  \Lambda_{\ell \ell}^2(\bx,\by) \, \frac{1}{Z_V(\bx)} \exp(-V(\bx,\by)) \, d\by \right)^{1/2} u_{i\ell} \notag \\
    &= \frac{1}{2}\sum_{j} \sum_{\ell} u_{j\ell} \left(  \int_{\R^m}  \Lambda_{\ell \ell}^2(\bx,\by) \, \frac{1}{Z_V(\bx)} \exp(-V(\bx,\by)) \, d\by \right)^{-1/2} \notag \\
    &\qquad \qquad{}\cdot (T_1^{\ell j}(\bx) + T_2^{\ell j}(\bx) + T_3^{\ell j}(\bx)) \, u_{i\ell},
\end{align*}
where, for clarity, we have split the derivative in three distinct contributions,
\begin{align*}
    T_1^{\ell j}(\bx) &=  \int_{\R^m}  \partial_{x_j} \left( \Lambda_{\ell \ell}^2(\bx,\by)\right) \, \frac{1}{Z_V(\bx)} \exp(-V(\bx,\by)) \, d\by, \\
    T_2^{\ell j}(\bx) &=  \frac{\partial_{x_j} Z_V(\bx)}{Z_V(\bx)} \int_{\R^m} \Lambda_{\ell \ell}^2(\bx,\by) \, \frac{1}{Z_V(\bx)} \exp(-V(\bx,\by)) \, d\by \intertext{and}
    T_3^{\ell j}(\bx) &= - \int_{\R^m} \Lambda_{\ell \ell}^2(\bx,\by) \, \partial_{x_j} V(\bx,\by) \, \frac{1}{Z_V(\bx)} \exp(-V(\bx,\by)) \, d\by.
\end{align*}
This splitting allows us to organize the terms in the left-hand side of \eqref{VeriKlimCond} into three corresponding vectors,
\begin{equation*}
 2 \bar{\sigma}_1(\bx) \nabla \cdot \bar{\sigma}_1(\bx)^T = B_1(\bx) + B_2(\bx) + B_3(\bx).
\end{equation*}
We compute the $i$th component of $B_1$ explicitly,
\begin{align*}
  (B_1(\bx))_i &= \sum_j \sum_{\ell} u_{i\ell} \left(  \int_{\R^m}  \Lambda_{\ell \ell}^2(\bx,\by) \, \frac{1}{Z_V(\bx)} \exp(-V(\bx,\by)) \, d\by \right)^{1/2} u_{j\ell} \notag \\
  &\quad \cdot \sum_{r} \sum_{\ell'} u_{r \ell'} \left(  \int_{\R^m}  \Lambda_{\ell' \ell'}^2(\bx,\by) \, \frac{1}{Z_V(\bx)} \exp(-V(\bx,\by)) \, d\by \right)^{-1/2} (T_1^{ \ell' r}(\bx)) \, u_{j \ell'}.
\end{align*}
By the orthogonality of $U$, we have the identity
\begin{equation*}
    \delta_{k \ell}= \sum_{s} u_{s k} u_{s \ell},
\end{equation*}
which allows us to simplify the expression for $(B_1(\bx))_i$,
\begin{align*}
  (B_1(\bx))_i &=  \sum_{\ell} \sum_{r} u_{i\ell}  T_1^{\ell r}(\bx)    u_{r \ell}.
\end{align*}
The same reasoning applies to the other two terms, giving
\begin{align*}
  (B_2(\bx))_i &=  \sum_{\ell} \sum_{r} u_{i\ell}  T_2^{\ell r}(\bx)    u_{r \ell}, \\
  (B_3(\bx))_i &=  \sum_{\ell} \sum_{r} u_{i\ell}  T_3^{\ell r}(\bx)    u_{r \ell}.
\end{align*}

Now, let us compute the $i$th component of the left hand side of \eqref{VeriKlimCond} directly,
\begin{multline*}
\left(    \nabla \cdot \left(  \bar{\sigma}_1(\bx) \, \bar{\sigma}_1^T(\bx) \right) \right)_i  
  = \sum_{r} \partial_{x_r} \sum_{\ell} u_{r \ell} \left(  \int_{\R^m}  \Lambda_{\ell \ell}^2(\bx,\by) \, \frac{1}{Z_V(\bx)} \exp(-V(\bx,\by)) \, d\by \right) u_{i\ell} \notag \\
  = \sum_{r}  \sum_{\ell} u_{r \ell} \left(  \int_{\R^m}  \partial_{x_r} \left(\Lambda_{\ell \ell}^2(\bx,\by) \right) \, \frac{1}{Z_V(\bx)} \exp(-V(\bx,\by)) \, d\by \right) u_{i\ell} \notag \\
  \quad + \sum_{r}  \sum_{\ell} u_{r \ell} \frac{\partial_{x_r}Z_V(\bx)}{Z_V(\bx)} \left(  \int_{\R^m}  \Lambda_{\ell \ell}^2(\bx,\by) \, \frac{1}{Z_V(\bx)} \exp(-V(\bx,\by)) \, d\by \right) u_{i\ell} \notag \\
  \quad - \sum_{r}  \sum_{\ell} u_{r \ell} \left(  \int_{\R^m}  \Lambda_{\ell \ell}^2(\bx,\by) \, \partial_{x_r}\left( V(\bx,\by)\right) \, \frac{1}{Z_V(\bx)} \exp(-V(\bx,\by)) \, d\by \right) u_{i\ell} \notag \\
  = (B_1(\bx)+B_2(\bx)+B_3(\bx))_i,
\end{multline*}
where each term matches precisely the corresponding $B_j(\bx)$ above.

This computation confirms that, for this physically relevant class of noise coefficients, the Klimontovich reversibility condition is indeed preserved under averaging. Thus, the independence and structural properties of the noise are maintained. This feature of the Klimontovich noise might make it a more natural choice for modelling processes which appear as limit of a coarse-graining procedure.

\paragraph{Acknowledgments} G.A.P. is partially supported by an ERC-EPSRC Frontier Research Guarantee through Grant No. EP/X038645, ERC Advanced Grant No. 247031 and a Leverhulme Trust Senior Research Fellowship, SRF$\backslash$R1$\backslash$241055. J.Z. gratefully acknowledges funding by the US Army Research Office (grant W911NF2310230). 

\appendix
\section{Conversion between different noise interpretations}
\label{sec:Conversion}

Given an SDE with multiplicative noise, the choice of noise interpretation (i.e., the evaluation point for the stochastic integral) affects the drift term in the equation. However, it is always possible to translate an SDE written in one interpretation into an equivalent SDE in another interpretation by appropriately modifying the drift. Here we state a well-know conversion result to translate between noises in different formulation. Particular examples are the It\^o  noise with $\lambda=0$, the Stratonovich noise with $\lambda=\tfrac{1}{2}$ and the Klimontovich noise with $\lambda=1$. The correction term in\eqref{DriftDictionary} below vanishes when $\sigma$ is constant, in which case all interpretations coincide.

Let $\lambda, \gamma \in [0,1]$ denote two possible interpretations of noise, as described previously. Consider the SDE
\begin{equation*}
    d X_t = B(X_t)\, dt + \sqrt{2}\, \sigma(X_t)\, \circ_\lambda d W_t,
\end{equation*}
where $\circ_\lambda$ denotes the stochastic integral evaluated with parameter $\lambda$. One can rewrite this SDE in the $\gamma$-interpretation as
\begin{equation*}
    d X_t = \widetilde{B}(X_t)\, dt + \sqrt{2}\, \sigma(X_t)\, \circ_\gamma d W_t,
\end{equation*}
where the drift $\widetilde{B}$ is given by
\begin{equation}\label{DriftDictionary}
    \widetilde{B}(x) = B(x) + 2(\lambda - \gamma)\, \left( \nabla \cdot \left[\sigma(x)\sigma^T(x)\right] - \sigma(x)\, \nabla \cdot \sigma^T(x) \right);
\end{equation}
for the Stratonovich-to-It\^o case, $\lambda = \frac 1 2$ and $\gamma = 0$, see~\cite[(3.31)]{pavliotis2014stochastic}; the other cases are by linearity in the time parameter $\lambda$ analogous.

\paragraph{Example: Stratonovich to It\^o.}  
For instance, converting from the Stratonovich interpretation ($\lambda = \tfrac{1}{2}$) to the It\^o interpretation ($\gamma = 0$), we obtain
\begin{equation*}
    d X_t = B(X_t)\, dt + \sqrt{2}\, \sigma(X_t)\, \circ_{S} d W_t
\end{equation*}
is equivalent to
\begin{equation*}
    d X_t = \left( B(X_t) +  \nabla \cdot (\sigma(X_t)\sigma^T(X_t)) -  \sigma(X_t)\, \nabla \cdot \sigma^T(X_t) \right) dt + \sqrt{2}\, \sigma(X_t)\, d W_t,
\end{equation*}
where the second SDE is written in the It\^o sense.

\begin{remark}
This formula provides a dictionary for translating between the most common conventions: It\^o ($\lambda=0$), Stratonovich ($\lambda=\tfrac{1}{2}$), and Klimontovich ($\lambda=1$). The correction term vanishes when $\sigma$ is constant, in which case all interpretations coincide.
\end{remark}

\section{Mosco convergence \`a la Kuwae-Shioya}
\label{appendix.mosco}
Here we briefly review the framework of Mosco convergence for Dirichlet forms on varying Hilbert spaces, following Kuwae–Shioya~\cite{kuwae2003convergence}, and the refinement by Kolesnikov~\cite{kolesnikov2006mosco}. We recall the key definitions of Hilbert space convergence, strong and weak convergence of vectors, and Mosco convergence of quadratic forms, as well as their connections to the limiting behavior of associated semigroups. These concepts underpin the analysis of averaging phenomena discussed in the main text.

\subsection{Convergence of Hilbert spaces}
The convergence of Dirichlet forms we are interested in takes place in the setting of convergence of a sequence of Hilbert spaces. We recall this notion of convergence introduced in~\cite{kuwae2003convergence}.

\begin{definition}[Convergence of Hilbert spaces]\label{HilConv}
A sequence of Hilbert spaces $\{ H_n \}_{n \geq 0}$ converges to a Hilbert space $H$ if there exist a dense subset $\Cfrak \subseteq H$ and a family of linear maps $\left\{ \Phi_n \colon \Cfrak \to H_n \right\}_n$ such that
\begin{equation*}
    \lim_{n \to \infty} \| \Phi_n f \|_{H_n} = \| f \|_{H}, \qquad \text{  for all } f \in \Cfrak.
\end{equation*}
\end{definition}

It is also necessary to introduce the concepts of strong and weak convergence of vectors living in a convergent sequence of Hilbert spaces. Hence in Definitions~\ref{strongcon}, \ref{weakcon} and~\ref{MoscoDef} below we assume that the spaces $\{ H_n \}_{n \geq 0}$ converge to the space $H$, in the sense just defined,  with the dense set  $\Cfrak \subset H$ and the sequence of operators $\{ \Phi_n \colon \Cfrak  \to H_n \}_n$ witnessing the convergence.

\bd[Strong convergence on Hilbert spaces]\label{strongcon}
A sequence of vectors $\{  f_n \}$ with $f_n$ in $H_n$ is said to \emph{strongly converge} to a vector $f \in H$ if there exists a sequence $\left\{ \tilde{f}_M \right\} \in \Cfrak$ such that
\be
\lim_{M\to \infty} \left\| \tilde{f}_M -f \right\|_{H} = 0 \nn
\ee
and
\be
\lim_{M \to \infty} \limsup_{n \to \infty} \left\| \Phi_n \tilde{f}_M -f_n \right\|_{H_n} = 0. \nn 
\ee
\ed
\bd[Weak convergence on Hilbert spaces]\label{weakcon}
A sequence of vectors $\{  f_n \}$ with $f_n \in H_n$ is said to \emph{weakly converge}  to a vector $f$ in a  Hilbert space $H$ if
\be
\lim_{n\to \infty} \left \langle f_n, g_n \right \rangle_{H_n} = \ \left \langle f, g \right \rangle_{H}, \nn 
\ee
for every sequence $\{g_n \}$ strongly convergent to $g \in H$.
\ed
\br\label{StrongConvPhin}
Notice that, as expected, strong convergence implies weak convergence, and, for any $f \in \Cfrak$, the sequence $\Phi_n f $ strongly-converges to $f$.
\er

\subsection{Definition of Mosco convergence}
In this section we assume the Hilbert convergence of a sequence of Hilbert spaces $\{ H_n \}_n$ to a space $H$.

\bd[Mosco convergence]\label{MoscoDef}
A sequence of Dirichlet forms $\{ (\caE_n, D(\caE_n))\}_n $, defined on Hilbert spaces $H_n$ \emph{Mosco converges} to a Dirichlet form $(\caE, D(\caE)) $, defined on some Hilbert space $H$, if the following two conditions hold.
\begin{description}
\item[Mosco I.] For every sequence of $f_n \in H_n$ weakly-converging  to $f$ in $H$, one has
\be
\caE ( f ) \leq \liminf_{n \to \infty} \caE_n ( f_n ). \nn 
\ee
\item[Mosco II.] For every $f \in H$, there exists a sequence $ f_n \in H_n$ strongly-converging  to $f$ in $H$ such that
\be
\caE ( f) = \lim_{n \to \infty} \caE_n ( f_n ). \nn 
\ee
\end{description}
\ed

The following theorem from \cite{kuwae2003convergence}, which relates Mosco convergence with convergence of semigroups and resolvents. 

\bt\label{MKS}
Let $\left\{ (\caE_n, D(\caE_n))\right\}_n $ be a sequence of Dirichlet forms on Hilbert spaces $H_n$ and let  $(\caE, D(\caE)) $ be a Dirichlet form  in some Hilbert space $H$. The following statements are equivalent:
\ben
\item $\left\{ (\caE_n, D(\caE_n))\right\}_n $ Mosco-converges to $\{ (\caE, D(\caE))\} $.
\item The associated sequence of semigroups $\left\{ T_{n} (t) \right\}_n $ strongly-converges to the semigroup $ T(t)$ for every $t >0$.
\een
\et

\subsection{Convergence \`a la Kolesnikov}
Here we include a notion of convergence introduced in~\cite{kolesnikov2006mosco}. This notion builds on the notion of Kuwae-Shioya from~\cite{kuwae2003convergence}, and it is immediately applicable in our setting.

Let us denote by $\caH$ the disjoint union of the converging sequence of Hilbert spaces $H_n$, i.e., $\caH := \cup_{n \in \N} H_n$ and define weak and strong convergence in $\caH$ as in Definition~\ref{weakcon} and Definition~\ref{strongcon}.

\begin{definition}[Kolesnikov convergence]\label{KolesnikovConv}
    Suppose that we are given a convergent sequence of Hilbert spaces $H_n \to H$, and a sequence of quadratic forms $\lbrace \xi_n\colon H_n \to \bar{\R} \rbrace$. We say that a sequence of pairs $\lbrace (\xi_n, H_n) \rbrace$ \emph{converges} to $\lbrace (\xi, H) \rbrace$ if
    \begin{equation*}
        \Phi_n(\Cfrak) \subset  D(\xi_n) \quad \text{ for every $n$,}
    \end{equation*}
    where $\Cfrak \subset D(\xi)$ is dense in $\left(D(\xi), \xi^1\right)$ and
    \begin{equation*}
        \lim_{n \to \infty} \xi_n(\Phi_n(f)) = \xi(f)
    \end{equation*}
    for every $f \in \Cfrak$.
\end{definition}
Given a convergent sequence $\lbrace (\xi_n, H_n) \rbrace$ we can consider the sequence of domains $D(\xi_n)$ as a further sequence of Hilbert spaces with inner products $\xi_n^1$. This sequence will naturally converge to $D(\xi)$. We then can also define the space analogous to $\caH$, we will denote this space by $\caH^1:= \cup_n D(\xi_n)$ and endow it with equivalent notions of strong and weak convergence as in Definition~\ref{weakcon} and Definition~\ref{strongcon}.

\begin{remark}
 Notice that in general we have $\caH \neq \caH^1$. What is true is that if $f_n$ converges strongly to some $f$ in $\caH^1$, it also converges strongly in $\caH$. This is false for weak convergence.    
\end{remark} 

\begin{theorem}\label{KolesnikovThm}
    Let the sequence $\lbrace (\xi_n, H_n) \rbrace$ converge to $\lbrace (\xi, H) \rbrace$ in the sense of Definition~\ref{KolesnikovConv}. Suppose that for every $\caH$-weakly convergent sequence $f_n \to f$ such that $\sup_n \xi_n^1(f_n) < \infty$ one has that $f \in D(\xi)$ and $f_n \to f$ in $\caH^1$. Then $\xi_n \to \xi$ in the Mosco sense.
\end{theorem}

\begin{remark}
    Notice that the conditions of Theorem \ref{KolesnikovThm} are trivially satisfied whenever it is true that $\caH =\caH^1$.
\end{remark}

\printbibliography

\end{document}